\documentclass[11pt]{article}
\usepackage{mathrsfs}
\usepackage{amsfonts}
\usepackage{authblk}
\usepackage{tikz}
\usetikzlibrary{positioning,arrows.meta}
\usepackage{makeidx}
\usepackage[leqno]{amsmath}
\usepackage{graphicx}
\usepackage{latexsym}
\usepackage{amsmath,amsfonts,amssymb,amsthm,mathrsfs,euscript,makeidx,color}
\usepackage{enumerate}

\allowdisplaybreaks[2]

\def\sqr#1#2{{\vcenter{\vbox{\hrule height.#2pt
				\hbox{\vrule width.#2pt height#1pt \kern#1pt \vrule width.#2pt}
				\hrule height.#2pt}}}}
\def\5n{\negthinspace \negthinspace \negthinspace \negthinspace \negthinspace }
\def\4n{\negthinspace \negthinspace \negthinspace \negthinspace }
\def\3n{\negthinspace \negthinspace \negthinspace }
\def\2n{\negthinspace \negthinspace }
\def\1n{\negthinspace }

\def\dbE{\mathbb{E}}
\def\dbF{\mathbb{F}}
\def\dbG{\mathbb{G}}

\def\dbP{\mathbb{P}}

\def\dbR{\mathbb{R}}
\def\dbS{\mathbb{S}}

\def\sD{\mathscr{D}}

\def\cF{\mathscr{T}}
\def\sU{\mathscr{U}}

\def\cF{{\cal F}}
\def\cG{{\cal G}}

\def\cN{{\cal N}}

\def\cU{{\cal U}}

\def\cX{{\cal X}}

\def\BS{{\bf S}}

\def\BU{{\bf U}}

\def\Im{{\mathop{\rm Im}\,}}
\def\Ker{\mbox{\rm Ker}}

\def\ds{\displaystyle}

\def\ss{\smallskip}
\def\ms{\medskip}

\def\q{\quad}
\def\qq{\qquad}
\def\hb{\hbox}

\def\({\Big (}
\def\){\Big )}

\def\lan{\langle}
\def\ran{\rangle}

\def\d{\delta}
\def\e{\varepsilon}

\def\f{\varphi}

\def\i{\infty}

\def\Th{\Theta}

\def\O{\Omega}

\def\cd{\cdot}

\def\bde{\begin{definition}\label}
\def\ede{\end{definition}}
\def\bel{\begin{equation}\label}
\def\ee{\end{equation}}
\def\bt{\begin{theorem}\label}
\def\et{\end{theorem}}
\def\bc{\begin{corollary}\label}
\def\ec{\end{corollary}}
\def\bl{\begin{lemma}\label}
\def\el{\end{lemma}}
\def\bp{\begin{proposition}\label}
\def\ep{\end{proposition}}
\def\bas{\begin{assumption}\label}
\def\eas{\end{assumption}}
\def\br{\begin{remark}\label}
\def\er{\end{remark}}
\def\bex{\begin{example}\label}
\def\ex{\end{example}}
\def\ba{\begin{array}}
\def\ea{\end{array}}
\def\ben{\begin{enumerate}}
\def\een{\end{enumerate}}

\def\square#1{\vbox{\hrule\hbox{\vrule height#1%
\kern#1\vrule}\hrule}}
\def\rectangle#1#2{\vbox{\hrule\hbox{\vrule height#1%
\kern#2\vrule}\hrule}}

\newtheorem{theorem}{\indent Theorem}[section]
\newtheorem{definition}[theorem]{\indent Definition}
\newtheorem{proposition}[theorem]{\indent Proposition}
\newtheorem{corollary}[theorem]{\indent Corollary}
\newtheorem{lemma}[theorem]{\indent Lemma}
\newtheorem{remark}[theorem]{\indent Remark}
\newtheorem{example}[theorem]{\indent Example}

\newtheorem{assumption}[theorem]{\indent Assumption}

\makeatletter

\@addtoreset{equation}{section}
\makeatother
	
\def\bea{\begin{equation*}}
\def\eea{\end{equation*}}
\def\bel{\begin{equation}\label}
\def\eel{\end{equation}}
\def\ba{\begin{array}}
\def\ea{\end{array}}

\def\({\Big (}
\def\){\Big )}

\def\q{\quad}
\def\qq{\qquad}

\def\d{\delta}
\def\e{\varepsilon}

\def\ms{\vspace{0.2cm}}
\def\ds{\displaystyle}

\begin{document}

\title{
A Unified Theory of Solvability for Stochastic Linear-Quadratic Control Problems I: The Homogeneous Case}
\author{Hongwei Mei\thanks{ Department of Mathematics and Statistics, Texas Tech University, Lubbock, TX 79409, USA; email: {\tt hongwei.mei@ttu.edu.} The author is partially supported by Simons Foundation's Travel Support for Mathematicians Program (No. 00002835).}}

\maketitle
\begin{abstract} The existing theory of stochastic linear-quadratic (LQ) optimal control problems focuses on open-loop and closed-loop solvability.  For finite horizon problems, it is widely accepted  that those two types of solvability are not equivalent in general,  as demonstrated by explicit counterexamples.   In contrast, they coincide in standard infinite-horizon formulations with constant coefficients under suitable stabilizability and integrability assumptions. This paper aims at challenging this well-established paradigm by developing a novel framework of closed-loop solvability,  namely generalized closed-loop solvability, which removes the   $L^2$-integrability constraint on the feedback strategies from the existing formulation. 
 It will be proven that   open-loop and generalized closed-loop solvability are equivalent and the classical closed-loop solvablity is a  particular case only when the admissible feedback strategy  is  $L^2$-integrable over the time horizon. In other words, the previously established non-equivalence between open-loop and closed-loop solvability stems solely from the limitations of the classical formulation which requires  the  $L^2$-integrability of the admissible feedback strategies over the time-horizon. Therefore, this  paper unifies the theory of the solvabilitiy   for stochastic LQ control problems.  Moreover, our results  offer a novel insight to stochastic LQ control problems: open-loop/generalized closed-loop solvability is essentially equivalent to piecewisely local  closed-loop solvability in the classical sense.    The current paper is focused on  homogeneous cases only and  non-homogeneous cases will be studied in a subsequent paper.
\end{abstract}

\ms
		
		\noindent\bf Keywords: \rm  Linear-quadratic control, generalized closed-loop solvability,
open-loop solvability, closed-loop solvability, Riccati equation.
		
		\ms
		
		\noindent\bf AMS Mathematics Subject Classification. \rm  49N10, 49N35, 93E20.
\section{Introduction}

The study of stochastic linear-quadratic (LQ) optimal control problems dates back to the pioneering work of Wonham \cite{Wonham1968}. Other foundational contributions to LQ control theory include those of Kalman \cite{Kalman1960} and Davis \cite{Davis1977}. In the classical setting, suitable positive-definiteness conditions ensure the existence of a unique optimal control admitting a linear state-feedback representation. Chen, Li, and Zhou \cite{Chen-Li-Zhou1998} demonstrated that stochastic LQ problems may remain solvable even when these conditions fail, opening the way to the study of indefinite stochastic LQ control. Subsequent developments include the works of Lim and Zhou \cite{Lim-Zhou1999}, Chen and Zhou \cite{Chen-Zhou2000}, and Chen and Yong \cite{Chen-Yong2000}. The solvability of indefinite Riccati equations under various technical conditions was further investigated by Hu and Zhou \cite{Hu-Zhou2003} and Qian and Zhou \cite{Qian-Zhou2013}.

Ait Rami, Moore, and Zhou \cite{Rami-Moore-Zhou2002} introduced a generalized Riccati equation involving the Moore--Penrose pseudoinverse and an additional algebraic constraint. Related contributions include the work of Ait Rami, Zhou, and Moore \cite{Rami-Zhou-Moore2002} on infinite-horizon stochastic LQ control and that of Wu and Zhou \cite{Wu-Zhou2001}. Huang, Li, and Yong \cite{Huang-Li-Yong-2015}, and Sun \cite{Sun-2017} subsequently studied infinite-horizon mean-field LQ optimal control problems. For stochastic LQ problems with random coefficients, we refer the reader to Chen and Yong \cite{Chen-Yong2000}, Tang \cite{Tang2003,Tang2015}, Du \cite{Du2015},  Kohlmann and Tang \cite{Kohlmann-Tang-2003}, Lu \cite{Lu-2020}.

A central issue in stochastic LQ control is the distinction between open-loop and closed-loop solvability. Sun, Li, and Yong \cite{Sun-Li-Yong-2016} established necessary and sufficient conditions for both solvability in the finite-horizon setting. Specifically, open-loop solvability is characterized by convexity of the cost functional together with the existence of a solution to the associated optimality system satisfying the stationarity condition, whereas closed-loop solvability is characterized by the existence of a regular solution to the associated differential Riccati equation and backward stochastic differential equation. They also provided examples showing that open-loop solvability is necessary but not sufficient for closed-loop solvability. In contrast, Sun and Yong \cite{Sun-Yong-2018} established the equivalence of open-loop and closed-loop solvability for infinite-horizon stochastic LQ problems under the standing assumptions of their framework. For a systematic treatment of these topics in stochastic LQ control and differential games, we refer the reader to the book \cite{Sun-Yong-2020}. Further developments for stochastic LQ problems with regime switching can be found in \cite{Zhang-Li-Xiong2021,Mei-Wei-Yong-2024,Mei-Wei-Yong-2025,Mei-Wang-Wei-Yong-2026}.

 To illustrate  the non-equivalence among open-loop and closed-loop solvability in a  more transparent way, let us present the following example within the previous framework of stochastic LQ optimal control.

\begin{example}\rm 
Consider 1-dimensional ordinary differential equation (ODE) with control on $[0,1]$:
$$dX(t)=u(t)dt\text{ with } X(s)=x.$$
The control process $u(\cd)$ is taken from the set of all admissible control:
$$\sU[s,T]=\Big\{u(\cd):[s,1]\rightarrow \dbR~\Big|~\int_s^T |u(t)|^2dt<\infty\Big\}.$$
Define the cost functional by
$$J_T^0 (s,x;u(\cd))=\int_s^1|(1-t) u(t)+X(t)|^2dt+|X(1)|^2.$$

The open-solvability aims to find an optimal $\bar u(\cd)\in\sU[s,1]$ corresponding to the initial $X(s)=x$ such that 
$$J_T^0 (s,x;\bar u(\cd))=\inf_{u(\cd)\in \sU[s,T]}J_T^0 (s,x;u(\cd)).$$
While the closed-solvability aims to find a uniform $\Th(\cd)\in L^2[0,1]$
such that 
$$J_T^0 (s,x; u^\Th(\cd))=\inf_{u(\cd)\in \sU[s,T]}J_T^0 (s,x;u(\cd))\text{ for all }X(s)=x.$$
Here $u^\Th(\cd)$ denotes the realized control using the feedback control $u(\cd)=\Th(\cd)X(\cd)$.

Under this framework, it is standard to see that the above problem is (uniquely) open-loop solvable. More specifically, for any $(s,x)\in[0,1]\times\dbR,$ $$\bar u(t)=-\frac x{1-s}\text{ with }\int_s^1|\bar u(t)|^2dt=\frac {x^2}{1-s}<\infty,$$
is the unique optimal control.
The optimal trajectory is 
$$\bar X(t)=x\frac{1-t}{1-s}=-(1-t)\bar u(t).$$

While the above problem is not closed-loop solvable.
 Since the open-loop optimal control is unique, if the problem is closed-loop solvable, it is necessary that $$\Th(t)=-\frac 1{1-t},\text{ while } \int_s^1(1-t)^{-2}dt=\infty.$$
 Hence the above problem is not closed-loop solvable. This example justifies the non-equivalence among closed-loop and open-loop solvability for stochastic LQ problems in a finite-horizon.

\end{example}
  Even though the above problem fails to be closed-loop solvable,  we notice that the open-loop optimal controls indeed can be realized by  a uniform feedback form using $\Th(t)=-(1-t)^{-1}$.  The failure of closed-loop solvability arises from the lack of $L^2$-integrability for the feedback  strategy only. This leads to the following question: {\bf Is it  possible to generalize the previous formulation of closed-loop solvability beyond the constraint of $L^2$-integrability so that the closed-loop and open-loop solvability become equivalent?}

The paper  provides a positive answer to the above question and hence unifies the theory of solvabilities for stochastic LQ control problems in  finite and infinite horizons. More specifically,  we develop a new framework of closed-loop solvability, namely {\it generalized closed-loop solvability}, where the $L^2$-integrability of the closed-loop strategy is not required necessarily. Under the new framework, we show that the open-loop solvability and generalized closed-loop solvability are equivalent.  Our result demonstrates that the non-equivalence stems solely from the technical restriction  of $L^2$-integrability of   closed-loop strategies in the earlier formulation of closed-loop solvability.    Moreover, our result shows that open-loop solvability leads to the existence a  piecewisely locally $L^2$-integrable closed-loop strategy for finite-horizon problems. The current paper is focused on  homogeneous problems  and  non-homogeneous cases will be studied in a subsequent paper.  Therefore, we mainly tackle the singularity of the associated differential Riccati equation (DRE) brought by the lack of $L^2$-integrability of the optimal feedback strategy in this paper.

The rest of this paper is arranged as follows. We formulate the main problem in Section \ref{sec:for}. Particularly, we recall some important results  on the classical theory of solvability. Then the new framework on generalized closed-loop solvability is introduced in Section \ref{sec:gcl}. Section \ref{sec:rp} introduces some regularized problems and Section \ref{sec:bp} proves the Bellman principle and a range condition, under a finiteness condition. Then our main result, the equivalence between open-loop and closed-loop solvability, is proven in Section \ref{sec:mr}. In Section \ref{sec:IP}, we apply the new framework on infinite-horizon stochastic LQ control problems and  reveals  the essence behind the equivalence between closed-loop and open-loop solvability  in this case.  Finally, some concluding remarks are made in Section \ref{sec:cr}.

\section{Formulation}\label{sec:for}

Let  $(\Omega,\cF,\dbP)$ be a complete probability measure space, on which   there exists a standard Brownian motion $\{W(t):t\geq 0\}$. Denote by $\dbF=\{\cF_t:t\geq 0\}$, the augmentation of the natural filtration of the Brownian motion with all  $\dbP$-null sets. 

For an $A\in\dbR^{n\times m}$, denote $A^\top$ by its transpose.
Let  $\dbS^k$ denote the set of $k$-dimensional symmetric matrices, 
 $S^k_{\geq 0}$ (resp. $S^k_{>0}$) denote
the   $k$-dimensional symmetric non-negative definite (resp., positive-definite) matrices.  For any $R\geq 0$, denote by $R^\dagger$ its Moore-Penrose pseudoinverse. For more details, please check Appendix \ref{sec:app}.

The following notations are adopted throughout this paper.
\begin{align*}&L^2_{\cF_s}(\Omega;\dbR^k)=\Big\{\f:\O\to\dbR^k\Big|\f(\cd)\hb{ is $\cF_s$-measurable with }\dbE|\f|^2<\infty\Big\},\\
&L^2_\dbF(0,T;\dbR^k):=\Big\{\f:[0,T]\times\O\to\dbR^k\Big|\\
&\qq\qq\f(\cd)\hb{ is $\dbF$-adapted with }\dbE\int_0^T|\f(s)|^2ds<\i\Big\},\\
&S^2_{\mathbb{F}}(0,T;\mathbb{R}^k)=
\Big\{
\varphi:[0,T]\times\Omega\to\mathbb{R}^k\Big|\\
&\qq\qq
\varphi(\cdot)\ \text{is }\mathbb{F}\text{-adapted   with continuous paths and }
\dbE\Big[
\sup_{0\le t\le T}|\varphi(t)|^2
\Big]<\infty
\Big\},
\\
&AC(0,T;\dbS^k):=\Big\{\f:[0,T]\to\dbS^k\Big|\f(\cd)\hb{ is absolute continuous}\Big\}.
\end{align*}
Throughout this paper, the following is always assumed:
\begin{align*}&
A(\cd),\in L^1(0,T;\dbR^{n\times n}), ~B(\cd),\in L^2(0,T;\dbR^{n\times m}),\\
&C(\cd)\in L^2(0,T;\dbR^{n\times n}),  ~D(\cd)\in L^\infty(0,T;\dbR^{n\times m}),\\
 &Q(\cd)\in L^1(0,T;\dbS^{n}),~ R(\cd)\in L^\infty(0,T;\dbS^{m}),~  S(\cd)\in L^2(0,T;\dbR^{m\times n}),~G\in\dbS^n.
 \end{align*}

Consider the following $n$-dimensional homogeneous controlled  stochastic differential equation (SDE): 
\begin{align}\label{31SDE4-1}\begin{cases}  &dX(t)=[A(t)X(t) + B(t)u(t)  ]dt+[C(t)X(t)  + D(t)u(t) ]dW(t),~t\in[s,T] \\
 &X(s)=x.\end{cases}\end{align}
The  cost functional is defined by 
 \begin{align}\label{31cost4-1}  
  & J_T^0(s,x;u(\cd)) \nonumber\\
   &=\dbE\Big[\int_s^T\lan Q(t)X(t),X(t)\ran +2\lan S(t)X(t),u(t)\ran+\lan R(t)u(t),u(t)\ran dt+\lan GX(T), X(T)\ran\Big].
\end{align}

The $n$-dimensional solution process $X(\cd)$ is  referred to as the {\it  state process}  and  the $m$-dimensional stochastic process $u(\cd)$ is called the  {\it control process}. Write by {\it the set of all possible initials} and  {\it the set of admissible controls}
\begin{align*}\sD:=\Big\{(s,x):s\in[0,T], x\in\dbR^n\Big\},\qq\sU[s,T]:=L^2_\dbF(s,T;\dbR^m).
\end{align*}

For any $u(\cdot)\in \sU[s,T]$ and $(s,  x)\in\sD$, it is well-known that \eqref{31SDE4-1} admits a unique solution written by 
$$X (\cd;s, x; u(\cd))\in S^2_\dbF(s,T;\dbR^n)\subseteq L^2_\dbF(s,T;\dbR^n).$$ 
Consequently, the cost functional  $J_T^0(s, x;u(\cd))$ is well-defined and finite. This allows us to consider the following optimal control problem.\ms

\noindent {\bf Problem (LQ)$^0_{T}$:} for any $(s, x)\in\sD$, find a $\bar u(\cd)\in \sU[s,T]$ such that 
\begin{equation*}
J^0_{T}(s, x;\bar u(\cd))=\inf_{u(\cd)\in \sU[s,T]}J^0_{T}(s, x; u(\cd))=:V_T^0(s, x).
\end{equation*}
Here $V_T^0(\cd)$ is called the {\it value function} for Problem (LQ)$ _{T}$. \ms

Problem (LQ)$^0_{T}$ is usually referred to as a {\it  homogeneous stochastic LQ control problem in a finite-horizon.} Such a problem has been extensively studied in the literature (see \cite{Sun-Yong-2020} for example). The previous results are mainly focused on the following perspectives.

\begin{definition}[Finiteness] \rm Problem (LQ)$^0_{T}$ is called {\it finite at $(s, x)\in\sD$} if $$V_T^0(s, x)>-\infty.$$
Problem (LQ)$^0_{T}$ is called {\it finite} if it is finite at all $(s, x)\in\sD.$
\end{definition}
\begin{definition}[Open-Loop Solvability]\label{31defopn} \rm Problem (LQ)$^0_{T}$ is called {\it open-loop solvable at $(s, x)\in\sD$} if there exists a $\bar u(\cd)\in \sU[s,T]$ (depending on $(s, x)$) such that 
$$ J_T^0(s, x;\bar u(\cd))=V(s, x)>-\infty.$$ Problem (LQ)$^0_{T}$ is called {\it open-loop solvable} if it is open-loop solvable at all $(s, x)\in\sD$.
\end{definition}

\begin{definition}[$L^2$-Closed-Loop Strategy]\rm $\Th: [0,T]\rightarrow\dbR^{m\times n }$ is called a {\it $L^2$-closed-loop strategy} if $\Th(\cd)\in L^2(0,T;\dbR^{m\times n })$.  Write by $\BU[0,T]$ the set of all possible $L^2$-closed-loop strategies.
\end{definition}

 For any $\Th(\cd)\in\BU[0,T]$,  write by the realized control process
$u^\Th(\cd)$ using the feedback control $u(t)=\Th(t)X(t)$ in \eqref{31SDE4-1}. For any given $(s,x)\in\sD,$ it  follows that 
$u^\Th(\cd)\in\sU[s,T].$
With those notations, we are able to define  closed-loop solvability.

\begin{definition}[Closed-Loop Solvability]\rm Problem (LQ)$^0_{T}$ is called {\it closed-loop solvable} if there exists a uniform  $\Th(\cd)\in L^2(0,T;\dbR^{m\times n})$ such that 
$$V_T^0(s, x)=J_T^0(s,x;u^{\Th}(\cd))\text{ for all }(s,x)\in\sD.$$
\end{definition}

It has been proven that the closed-loop solvability can be euiqvalenltly characterized as the well-posedness of associated differential Riccati equation (DRE) in \cite{Sun-Li-Yong-2016}.

\begin{definition}[Well-posedness of DRE]\rm Consider the following DRE
\begin{align}\label{DRE}\begin{cases}
&\dot P+A^\top P+PA+C^\top P C+Q\\
&\qq- (B^\top P+D^\top PC +S)^\top(R+D^\top PD)^\dagger (B^\top P+D^\top PC +S) =0,
\\
& P(T)=G,\\
& R+D^\top P D\geq 0,~ \Im[B^\top P+D^\top PC +S]\subset \Im[R+D^\top P D],~ \text{a.s.-}dt.\end{cases}
\end{align}  
The definition of $\Im[\cdot]$ is referred to \eqref{defIM}.

(1). $P(\cd)\in AC(0,T;\dbS^n)$ is called a {\it strongly regular solution} if there exists a $\d>0$ such that 
$$R+D^\top P D\geq \d I, \text{ a.s.-}dt.$$
In this case, the Moore-Penrose pseudoinverse is the same as the  inverse in the classical sense, i.e. $(R+D^\top PD)^\dagger=(R+D^\top PD)^{-1}$.

(2). $P(\cd)\in AC(0,T;\dbS^n)$ is called a {\it $L^2$-regular solution} if  
$$\Th^\circ(\cd):=-(R+D^\top PD)^\dagger (B^\top P+D^\top PC +S)\in L^2(0,T;\dbR^{m\times n}).$$
\end{definition}

The following uniform convexity is well-adopted in the literature as well.

\begin{definition}[Uniform Convexity] \rm The cost functional  is called {\it uniform convex}  if there exists a $\d>0$ such that 
$$J_T^0(0,0;u(\cd))\geq \d\dbE\int_0^T|u(t)|^2dt\text{ for all }u(\cd)\in\sU[0,T].$$
\end{definition}

With all above concepts, the following theorem summarizes the results on the closed-loop solvability obtained in the literature (see \cite{Sun-Li-Yong-2016}).

\begin{theorem} \label{thm1}
{\rm (1).} If  Problem (LQ)$^0_{T}$ is finite, then $$J_T^0(s, 0;u(\cd))\geq 0\text{ for any }u(\cd)\in \sU[s,T].$$ In this case,  $J_T^0(s, 0;\cd)$ is convex on  $\sU[s,T]$.

{\rm (2).} Problem (LQ)$^0_{T}$ is closed-loop solvable if and only if DRE \eqref{DRE} admits an $L^2$-regular solution. In this case, $\Th(\cd)\in\BU[0,T]$ is an optimal  $L^2$-closed-loop strategy   if and only if $$\Th(\cd)=\Th^\circ(\cd)+\Phi(\cd)\in\BU[0,T]$$ for some $\Phi(\cd)\in\BU[0,T]$ with
$(R+D^\top P D)\Phi=0,\text{ a.s.-}dt.$

{\rm (3).} DRE \eqref{DRE} admits  a strongly  regular solution if and only if the uniform convexity condition holds.  In this case, Problem (LQ)$^0_{T}$ is  closed-loop solvable and $\Th^\circ(\cd)$ is the unique optimal $L^2$-closed-loop strategy. 

\end{theorem}

The above theorem will be used in the new framework developed below. We emphasize that the classical formulation of closed-loop solvability requires the feedback strategy $\Th(\cd)$ to be $L^2$-integrable, a condition closely tied to the well-posedness of the associated Riccati equation. In the following sections, we will show that this integrability requirement is a technical restriction that can be removed within our new framework.

\section{Generalized Closed-Loop Solvability}\label{sec:gcl}

As demonstrated in the introduction, we know that open-loop  and closed-loop solvability are not equivalent. The gap lies in the constraint on the $L^2$-integrability of the closed-loop strategies.  In this section, we will develop a new framework, namely {\it generalized closed-loop solvability}, to fill this gap.
 We need the following definition. 

\begin{definition}[Admissible Closed-Loop Strategy]\label{defacL} \rm
$\Th:[0,T]\rightarrow\dbR^{m\times n}$ is called by an {\it admissible closed-loop strategy} if for any $(s, x)\in\sD$, there exists a control process $ u(\cd)\in \sU[s,T]$ such that 
$$ u(\cd)=\Th(\cd)X(\cd;s, x; u(\cd))\text{ in } \sU[s,T].$$
Write $\BU^{ad}_0[0,T]$ by the set of all admissible  feedback strategies. We emphasize that $\Th(\cd)$ is independent of the initial time-state pair$(s,x)\in\sD.$
\end{definition}

Compared  with  $L^2$-closed-loop strategies, without assuming  the integrability of $\Th(\cd)$,  the new definition is constructed based on verifying the $L^2$-integrability of $\Th(\cd)X(\cd;s, x; u(\cd))$ as a whole.   We can steadily see that the feedback strategy $\Th(t)=-(1-t)^{-1}$ in the example presented in the introduction is  an admissible closed-loop strategy. We will show that the above definition essentially generalizes the notation of $u^\Th(\cd)$ for $\Th(\cd)\in \BU^{ad}_0[0,T]$ as a closed-loop strategy with  $$\BU_0[0,T]\subsetneq \BU^{ad}_0[0,T].$$ The following is the result.

\begin{proposition}\label{proposition Thlocal} Suppose $\Th(\cd)$ is an admissible closed-loop strategy. Then at any $(s,x)\in \sD,$ the following SDE admits a unique solution $X(\cd)\in S_\dbF^2(s,T;\dbR^n)$: \begin{align}\label{feedbackSDE}\begin{cases}  &dX=(A+B\Th )Xdt+(C+D\Th)XdW,~t\in[s,T], \\
 &X(s)=x.\end{cases}\end{align}
\end{proposition}

\begin{proof}
We first show that, for every $s<T$, there exists
$\varepsilon_r\in(0,T-s]$ such that
$$
\Theta(\cd)\in L^2(s,s+\varepsilon_s;\mathbb R^{m\times n}).
$$

Fix $s<T$. For each coordinate vector $e_j\in\mathbb R^n$,
choose a control $u_j\in L_{\mathbb F}^2(s,T;\mathbb R^m)$
satisfying the assumed feedback identity, and denote its
corresponding state by $X_j$. Set
$$
\cX(t)=(X_1(t),\cdots,X_n(t))
\text{ and }
\cU(t)=(u_1(t),\cdots,u_n(t)).
$$
Then it follows that 
$$
\mathcal X(s)=I,\qquad
\mathcal U(\cd)=\Theta(\cd)\mathcal X(\cd)
\quad \text{a.s.-}dt\times  d\mathbb P,
$$
and
$$
\mathbb E\int_s^T\Vert\mathcal U(t)\Vert^2\,dt<\infty.
$$

Note that  $\cX(\cd)\in S_\dbF^2(s,T;\dbR^{n\times n})$ satisfies
$$
\sup_{s\le t\le s+\varepsilon}
\Vert\mathcal X(t)-I\Vert
\rightarrow0
\quad\text{a.s. as }\varepsilon\rightarrow0^+.
$$
Thus, there exists a deterministic
$\varepsilon_s\in(0,T-s]$ such that $$\mathbb P(E_s)\ge\frac12
$$
where $$
E_s:=
\left\{
\sup_{s\le t\le s+\varepsilon_s}
\Vert\cX(t)-I\Vert\le\frac12
\right\}
$$

On $E_s$, the matrix $\mathcal X(s)$ is invertible for every
$t\in[s,s+\varepsilon_s]$, with
$
\Vert\cX(t)^{-1}\Vert\le2.
$
Using $\Theta(\cd)=\cU(\cd)\cX^{-1}(\cd)$ on this event, we obtain
$$
\mathbf 1_{E_s}\|\Theta(t)\Vert^2
\le4\mathbf 1_{E_s}\|\mathcal U(t)\Vert^2
\quad \text{a.s.-}dt\times d\dbP.
$$
Since $\Theta$ is deterministic, it follows that
$$
\frac12\int_s^{s+\varepsilon_s}\|\Theta(t)\Vert^2\,dt
\le
\mathbb P(E_s)\int_s^{s+\varepsilon_s}\|\Theta(t)\Vert^2\,dt
\le
4\mathbb E\int_s^{s+\varepsilon_s}
\Vert\mathcal U(t)\Vert^2\,dt<\infty.
$$
This proves the $\Th(\cd)\in L^2(s,s+\e_s;\dbR^{m\times n})$. Then there exists a unique solution $X(\cd;s,x;u^\Th(\cd))=\cX(\cd)x\in S_\dbF^2(s,s+\e_s;\dbR^{ n})$, which can be continuously extend to $t=s+\e_s$ in $L^2$-norm.

To extend such a representation to the whole time horizon $[s,T]$, 
suppose the solution exists uniquely  on $[s,s_*]$ only. If $s_*<T$, the local argument above implies the existence and uniqueness of the solution on
$[s_*,s_*+\varepsilon_{s_*}]$ for some $\varepsilon_{s_*}>0$. This is a contradiction. The proof is complete.
\end{proof}

 Indeed, we can further show that $\Th(\cd)$ is piecewisly $L^2$-integrable out of those singularity times.

\begin{proposition} For any $\Th(\cd)\in\BU^{ad}_0[0,T]$, there exists a decomposition  $\{[t_k,s_k): k=1,2\cdots\}$ of $[0,T)$ such that $\Th(\cd)$ is locally $L^2$-integrable on $[t_k,s_k)$ for all $k\geq 1,$ i.e. $$\int_{t_k}^{s_{k}-\e}|\Th(t)|^2dt<\infty,\text{ for any }\e\in(0,s_{k}-t_k).$$
\end{proposition}
\begin{proof}Define the open set
$$
O:=\bigcup_{s\in[0,T)}(s,s+\e_s)\subseteq(0,T).
$$
where $\e_s$ is the selected in Proposition \ref{proposition Thlocal}.  Then there exists at most countably many disjoint open intervals $\{(t_k,s_k):k\in K\}$ such that 
$
O=\bigcup_{k\in K}(t_k,s_k).
$

We claim that
$$
[0,T)=\bigcup_{k\in K}[t_k,s_k).
$$

First, it is easy to see that $O\subset\bigcup_{k\in K}[t_k,s_k)$.
Second, suppose  $s\in[0,T)\setminus O$, then
$(s,s+\e_s)\subseteq O$ lies in a single connected component
$(t_k,s_k)$. Because $s\notin O $, it is necessary that   $t_k=r$.  This leads to  $r\in[t_k,s_k)$. Note that $\{(t_k,s_k):k\in K\}$ is  mutually disjoint, $\{[t_k,s_k):k\in K\}$ is mutually disjoint as well. Therefore our claim is true.

Fix $k\in K$ and $\varepsilon\in(0,s_k-t_k)$. For any 
$c\in\bigl(t_k,\min(s_k-\varepsilon,t_k+\e_{t_k})\bigr),
$
it follows that 
$$
\int_{t_k}^{\min(s_k-\varepsilon,t_k+\e_{t_k})}
|\Theta(t)|^2dt<\i.
$$
Since $[\min(s_k-\varepsilon,t_k+\e_{t_k}),s_k-\e]$ is a compact subset of $O$, it is covered by
finitely many intervals
$
\{(r_j,r_j+h_{r_j}):j=1,\ldots,N\}.
$
Consequently,
\begin{align*}
\int_{t_k}^{s_k-\e}|\Theta|^2dt
&\leq\int_{t_k}^{\min(s_k-\varepsilon,t_k+\e_{t_k})}|\Theta|^2dt+
\sum_{j=1}^N\int_{r_j}^{r_j+h_{r_j}}|\Theta|^2dt<\infty.
\end{align*}
The proof is complete.
\end{proof}

Even thought it is possible that $$\int_{t_k}^{s_k}|\Th(t)|^2dt=\infty,$$
the definition of admissible closed-loop strategy guarantees that the solution of \eqref{feedbackSDE} starting at $(t_k,x)\in\sD$ can be continuously extended to $t=s_k$ with $\dbE|X(s_k)|^2<\infty$. This explains why \eqref{feedbackSDE} admits a solution in $S^2_\dbF(s,T;\dbR^n)$ for any $(s,x)\in\sD.$ Therefore,  we are allowed to write  $$X(\cd;s,x;u^\Th(\cd))\in S_\dbF^2(s,T;\dbR^n)\text{ for any }\Th(\cd)\in\BU^{ad}_0[0,T]\text{ and }(s,x)\in\sD.
$$
The above arguments also illustrate why  Definition \ref{defacL} is named by generalized closed-loop solvability, rather than strong open-loop solvability. To provide a clearer idea for $\Th(\cd)\in\BU^{ad}_0[0,T]$, let us present some representative   examples below.

\begin{example}\rm On $[0,1]$, consider 1-dimensional SDE $$dX(t)=\Th(t) X(t)[b(t)dt+\sigma(t)  dW(t)].$$

(1). When $b(t)=1$ and $\sigma(t)=0$, 
$\Th_1(t)=-1/(1-t)$ is admissible. While $\Th_2(t)=1/(1-t)$ and $\Th_3(t)=1/t$  are not admissible .

(2). If  $\sigma(t)=b$ and $\sigma(t)=1$ select a $\Th(\cd)$ such that $\int_0^{1-\e}|\Th(t)|^2dt<\infty$ for any $\e>0$, while  $\int_0^{1}|\Th(t)|^2dt=\infty$. The solution has the following presentation 
$$X(t)=x\exp\Big\{\int_s^t b\Th(s)-\frac{\sigma^2}2\Th^2(s)ds+\int_0^t \Th(s)dW(s)\Big\}, ~t\in[s,1),$$
given initial $X(s)=x.$ As $t\rightarrow 1^-$, $X(t)\rightarrow 0,$ a.s.. This says that   the solution can be continuously extended to time $t=1$ pathwisely.

While this is not sufficient for  $\Th(\cd)\in \BU^{ad}_0[0,1]$ since
$$\dbE\int_s^1|\Th(t)X(t)|^2dt=x^2\int_s^1\Th^2\exp\Big[\int_s^t2b\Th+\sigma^2\Th^2dr\Big]dt=\infty.$$

(3). Consider
$$
dX(t)=\Th(t)X(t)\Big(a dt+\sqrt{1-t}dW(t)\Big),~X(s)=x.
$$
Let $\Th(t)=-(1-t)^{-1}$. Detailed calculation shows that 
$$\dbE|X(t)|^2=x^2\Big(\frac{1-t}{1-s}\Big)^{2a-1}\text{ and }\dbE\int_0^1|\Th(t)X(t)|^2dt=\frac{x^2}{(1-s)^{2a-1}}\int_0^1(1-t)^{2a-3}dt.$$
This says that $\Th(\cd)\in\BU^{ad}_0[0,1]$ if $a>1.$
\end{example}

Now we are ready to define the generalized closed-loop solvability.
\begin{definition}[Generalized Closed-Loop Solvability]\rm 
  Problem (LQ)$_T^0$ is called {\it generalized closed-loop solvable} if  there exists a uniform $\Th(\cd)\in \BU^{ad}_0[0,T]$ such that  
$$V_T^0(s, x)=J_T^0(s,x;u^{\Th}(\cd))\text{ for all }(s,x)\in\sD.$$

\end{definition}

Provided that $\BU[0,T]\subset \BU^{ad}_0[0,T]$, if the optimal admissible closed-loop strategy $\Th$ belongs to $\BU[0,T]$, then generalized closed-loop solvablility refines to closed-loop solavbility. We are ready to  present  the main result of this paper.

 \begin{theorem}\label{31thmclosedloopsolv}
{\rm(1).}  Problem (LQ)$^0_T $ is generalized closed-loop solvable if and only if it is open-loop solvable. In this scenario, there exists an optimal admissible closed-loop strategy $\bar \Th^*(\cd)\in\BU^{ad}_0[0,T]$ such that $u^{\bar \Th^*}(\cd)$ is an optimal open-loop control at any $(s,x)$ satisfying 
$$\dbE\int_s^T|u^{\bar \Th^*}(t)|^2dt\leq \dbE\int_s^T| \bar u(t)|^2dt$$
for any  optimal open-loop control $\bar u(\cd)\in\sU[s,T]$ at $(s,x)\in\sD,$ i.e. $u^{\bar \Th^*}(\cd)$ is the optimal open-loop control at $(s,x)$ with the minimal $L^2$-norm. The value function
$V_T^0(s,x)=\lan\bar P(s)x,x\ran.$

{\rm(2).}  If Problem (LQ)$^0_T $ is open-loop/generalized closed-loop solvable,  there exists a decomposition  $\{[t_k,s_k): k\in K\}$ of $[0,T)$ such that $\bar \Th^*(\cd)$ is locally $L^2$-integrable on $[t_k,s_k)$ for all $k\in K,$ i.e.
$$\int_{t_k}^{s_k-\e}|\bar \Th^*(t)|^2dt<\infty\text{ for each }\e\in (0,s_k-t_k).$$
Here $K$ has at most countable many elements.
Consequently,  on each $[t_{k},s_k)$, $\bar P(\cd)\in AC(t_k,s_k-\e;\dbS^n)$ is the $L^2$-regular solution to 
\begin{align}\label{DREk}\begin{cases}
&\dot {\bar P}+A^\top \bar P+\bar PA+C^\top \bar P C+Q\\
&\q- (B^\top \bar P+D^\top \bar PC +S)^\top(R+D^\top \bar PD)^\dagger (B^\top \bar P+D^\top \bar PC +S) =0, ~~t\in[t_k,s_{k}-\e],
\\
&R+D^\top \bar P D\geq 0, ~\text{a.s.-}dt\end{cases}
\end{align} 
for any $\e\in (0,s_k-t_k).$ \ss

{\rm(3).} 
At each $(s,x)\in\sD$ with $s\in[t_k,s_k)$, it follows that  $\bar P(s_k)\geq \bar P(s_k^-)$ and
$\bar X(s_k)\in \Ker[\bar P(s_k)-\bar P(s_k^-)]\text{ a.s.-}d\dbP,$
where $\bar X(\cd)$ is any optimal trajectory starting from $(s,x)$.
\end{theorem}

Some explanation is worth here for (2) and (3). In (2), we provide a new characterization for open-loop solvability. When the problem is merely open-loop solvable, there exist a sequence of singular times where the $L^2$-norm of the optimal admissible closed-loop strategy explodes. At those singular times,  the value function can be  discontinuous. While during the time horizon between two singularities, Problem (LQ)$_T^0$ is admits at least an optimal $L^2$-closed-loop  strategy. This essentially says that   Problem (LQ)$_T^0$ is  piecewisely locally closed-loop solvable.

 Indeed, (3)  provides a formal justification for the optimality of $\bar\Th^*(\cd)$ through the application of It\^o's formula. For simplicity, we assume that the terminal time $T$ is the only singular time. We will show how to handle the discontinuity of the value function at $T$.

Note that $P(T^-)\leq P(T)$.
Let $X(\cd)$ be the solution process starting at $(s,x)\in \sD $ under   $u(\cd)\in\sU[s,T]$. Applying It\^o's formula on $\lan \bar P(t)X(t), X(t)\ran$, steady calculation yields that 
\begin{align*}&J_T^0(s,x;u(\cd))=\lan\bar Px,x\ran\\
&\q+\dbE\int_s^T|u-\bar \Th^* X|^2dt+\dbE\lan (P(T)-P(T^-))X(T), X(T)\ran\geq \lan\bar Px,x\ran.\end{align*}
The equality
is achieved at $u(\cd)=\bar u^{\bar\Th^*}(\cd)\in\sU[s,T]$ as  $\bar X^*(T)\in\Ker[P(T)-P(T^-)]$ provided in (3). This justifies the feedback strategy $u(\cd)=\Theta^*(\cd) X(t)$ is optimal.

It should be noted that the above can not be used as a proof for the optimality of $\bar\Th^*(\cd)$ because it is assumed that  $\bar X^*(\cd)$ is the optimal trajectory when justifying $\bar X^*(T)\in\Ker[P(T)-P(T^-)]$. A different approach of justifying the optimality is adopted in this paper, which consists of the main contribution of this paper. We also want to emphasize that $\bar\Th^*(\cd)$ can not be replaced by $\bar\Th^\circ(\cd)$ since the feedback strategy $\bar\Th^\circ(\cd)$ might not be admissible.
We also can see that at the singular time $T$, the optimal trajectory vanishes on the subspace generating by the discontinuity: $$\{x:\lan [P(T)-P(T^-)]x,x\ran>0,~ |x|=1\}\subset\dbR^n.$$
This reveals how the discontinuity of the value function is bypassed along the optimal trajectory for homogeneous cases.

For those readers who are familiar with infinite-horizon problems, it  can be seen that the definition of generalized closed-loop solvability is consistent with the definition of closed-loop solvability for infinite horizon problems. Particularly for infinite-horizon problem, the optimal strategy $\Th$ (which is also a stabilizer and is invariant with time) is not $L^2$-integrable over $[0,\i)$ while $\Th X(\cd) $ is $L^2$-integrable. Moreover, because $\Th$ is invariant in time, it is a $L^2$-closed-loop strategy on any $[0,T]$. Hence the closed-loop and open-loop solvability are naturally equivalent for infinite-horizon problems. Therefore, the above theory unifies the solvablity theory for stochastic  LQ control problems in finite or infinite horizons. More details will be presented in Section \ref{sec:IP}.

We will prove this main result in the  Section \ref{sec:mr} below. For readers' convenience, let us
 summarize the results obtained in the new framework in the following figure.

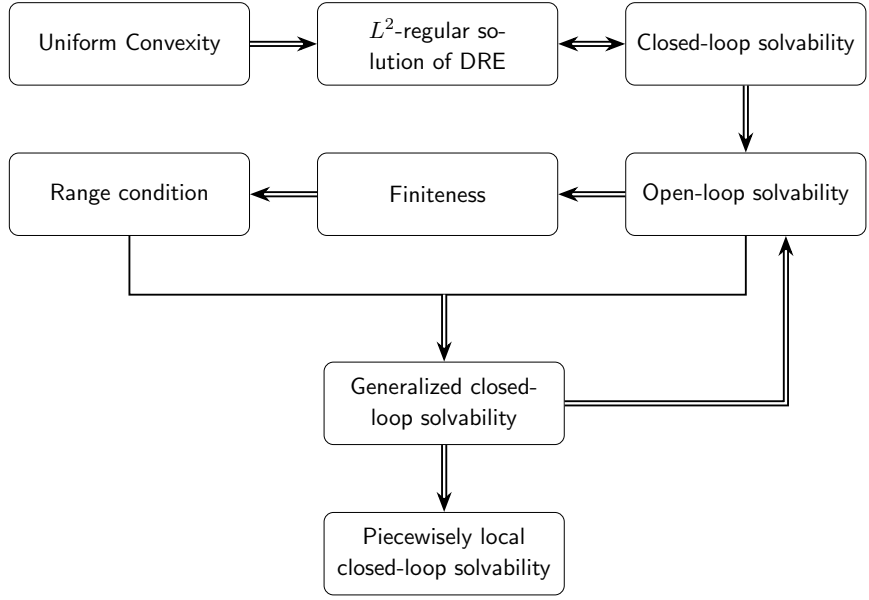
\begin{figure}[htbp]
\centering
\resizebox{0.7\textwidth}{!}{
\begin{tikzpicture}[
    >=Stealth,
    node distance=1.0cm and 1cm, 
    block/.style={
        draw,
        rectangle,
        rounded corners,
        minimum height=3.2em,
        text width=3.3cm,         
        align=center,
        font=\small\sffamily,
        inner sep=4pt             
    },
    terminal/.style={
        block,
        fill=green!5,
        text width=5cm,
        thick
    }
]


\coordinate (E) at (5.2,4.5) {};
\node[block, left=of E] (A) {Finiteness };
 \node[block, right=of A] (B) {Open-loop solvability};
\node[block, left=of A] (C) {Range condition};

\node[block, above=of B] (E) { Closed-loop solvability};
\node[block, above=of A] (F) { $L^2$-regular solution of DRE};
\node[block, above=of C] (G) { Uniform Convexity};
\coordinate (H) at (2.5,3) { };
\node[block, below=of H] (I) { Generalized closed-loop solvability};
\node[block, below=of I] (J) { Piecewisely local closed-loop solvability};
 \draw[double distance=1pt, -{Stealth[length=3mm]}, thick] (I) -- (J);
 \draw[double distance=1pt, -{Stealth[length=3mm]}, thick] (G) -- (F);
  \draw[double distance=1pt, {Stealth[length=3mm]}-{Stealth[length=3mm]}, thick] (F) -- (E);
   \draw[double distance=1pt, -{Stealth[length=3mm]}, thick] (E) -- (B);
   \draw[double distance=1pt, -{Stealth[length=3mm]}, thick] (B) -- (A);
   \draw[double distance=1pt, -{Stealth[length=3mm]}, thick] (A) -- (C);
     \draw[-, thick] (C) |- (H);
     \draw[-, thick] (B) |- (H);
     \draw[double distance=1pt, -{Stealth[length=3mm]}, thick] (H) -- (I);
     \draw[double distance=1pt, -{Stealth[length=3mm]}, thick] (I) -| ([xshift=3cm]B);
\end{tikzpicture}
}
\caption{Homogeneous  Problems over Finite-Horizon }
\label{f1}
\end{figure}

\section{Regularized Problems}\label{sec:rp}
In this section, we will work with  a sequence of regularized  problems, which are essential in our proof.  

By Theorem \ref{thm1}, we know that the finiteness leads to the convexity of the cost functional. To achieve the uniform convexity, we introduce the following regularized cost functional:
$$J_T^{0,\d}(s, x;u(\cd))=J_T^0(s, x;u(\cd))+\d\dbE\int_s^T|u(t)|^2dt.$$
 This allows us to consider the following regularized control problem:\ms

\noindent {\bf Problem (LQ)$^{0,\d}_{T}$:} for any $(s, x)\in\sD$, find a $\bar u^\d(\cd)\in \sU[s,T]$ such that 
\begin{equation*}
J^{\d}_{T}(s, x;\bar u^\d(\cd))=\inf_{u(\cd)\in \sU[s,T]}J^{\d}_{T}(s, x; u(\cd))=V^{\d}_{T}(s, x).
\end{equation*}

By Theorem \ref{thm1}, Problem (LQ)$^{0,\d}_{T}$ is closed-loop solvable with  the unique optimal  $L^2$-closed-loop strategy 
$$\bar \Th^\d(\cd)=-(R+D^\top {\bar P^\d} D)^{-1}(B^\top {\bar P^\d}+D^\top {\bar P^\d}C+S)\in \BU_0[0,T],$$
where $\bar P^\d(\cd)$ is the strongly regular solution to 
\begin{align}\label{31Riccatinot22}
\begin{cases}& \dot {\bar P}^\d+A^\top {\bar P^\d}+{\bar P^\d}A+C^\top {\bar P^\d} C+Q\\
&\q-(B^\top {\bar P^\d}+D^\top {\bar P^\d}C+S)^\top(R+\d I+D^\top {\bar P^\d} D)^{-1}(B^\top {\bar P^\d}+D^\top { {\bar P^\d}}C+S)=0,\\
&R+D^\top {\bar P^\d} D\geq \d I,~
{\bar P^\d}(T)=G.
\end{cases}
\end{align}

Now, we can show that the value of  Problem (LQ)$^0_{T}$ at any $(s, x)\in\sD$ can be approximated through the $L^2$-closed-loop strategies $\{\bar \Th_\d(\cd)\}$.

\begin{proposition} \label{31themfiniteasy}Suppose that Problem (LQ)$^0_{T}$  is finite. Then for any $(s, x)\in\sD$,
$$\lim_{\d\rightarrow 0^+} J_T^0(s, x; u^{\bar\Th^\d}(\cd))=\lim_{\d\rightarrow 0^+} V^{\d}_{T}(s, x)=V_T^0 (s, x).$$

\end{proposition}

\begin{proof}
At $(s, x)\in\sD$, suppose there exists a sequence $u_k(\cd)\in\sU[s,T]$ such that 
$$J_T^0(s, x; u_k(\cd))\leq  V_T^0(s, x)+\frac1k.$$
Then 
\begin{align*}&V^{\d}_{T}(s, x)\geq V_{T}(s, x)\geq J_T^0(s, x; u_k(\cd))-\frac1k\\
&\q\geq J_T^0(s, x; u_k(\cd))+\d\dbE\int_s^T|u_k(t)|^2dt-\Big(\d\dbE\int_s^T|u_k(t)|^2dt+\frac1k\Big)\\
&\q\geq V^{\d}_{T}(s, x)-\Big(\d\dbE\int_s^T|u_k(t)|^2dt+\frac1k\Big)
\end{align*}
Take $\d\rightarrow 0^+ $ and  $k\rightarrow\infty$ sequentially, it follows that 
$$V_{T}(s, x)\leq \lim_{\d\rightarrow 0^+} J_T^0(s, x; u^{\bar\Th^\d})= \lim_{\d\rightarrow 0^+} V^{\d}_{T}(s, x)=V_{T}(s, x).$$
The proof is complete.
\end{proof}
The above theorem unveils a way to identify a sequence of $L^2$-feedback strategies whose realizations are asymptotically optimal Problem (LQ)$_T^0$ at any initial.

If we further assume that Problem (LQ)$_T^0$ is open-loop solvable, we can identify one particular optimal control which admits the least $L^2$-norm among all possible optimal controls.

\begin{lemma} \label{31lemmmincontol}Suppose Problem (LQ)$_T^0$ is solvable at $(s, x)\in\sD$. Then there exists a $\d_k\downarrow 0$ such that  $\bar u^{\d_k}(\cd)\in\sU[s,T]$ admits a limit $\bar u^*(\cd)\in\sU[s,T]$ which is an optimal control Problem (LQ)$_T$. It also follows that 
$$\dbE\int_s^T |\bar u^*(t)|^2dt\leq \dbE\int_s^T |\bar u(t)|^2dt,$$
for any $\bar u(\cd)\in\sU[s,T]$ being optimal for Problem (LQ)$_T^0$ at $(s, x).$ Moreover, for each $(s, x)\in\sD$,
\begin{align}\label{31derivativePd}\lim_{\d\rightarrow0^+}\Big\lan \frac{ \bar P^\delta(s) -\bar P(s) }{\d}x,x\Big\ran=\dbE\int_s^T|\bar u^*(t)|^2dt.\end{align}
 
\end{lemma}

\begin{proof}
  Note that
\begin{align}\label{31PdP1}& \lan \bar P^\delta(s) x,x\ran-\lan \bar P(s) x,x\ran=V^{\d}(s, x)-V(s, x)\\
&\q=J_T^{0,\d}(s,  x; \bar u^{\d}(\cd))-J_T^0 (s, x; \bar u(\cd))\nonumber\\&\q
\leq  J_T^{0,\d}(s, x; \bar u(\cd))-J_T^0 (s, x; \bar u(\cd))=\delta\dbE \int_s^T|\bar u(t)|^2dt.\nonumber
\end{align}
and
\begin{align}\label{31PdP2}&\lan \bar P^\delta(s) x,x\ran-\lan \bar P(s) x,x\ran=V^{\d}(s, x)-V(s, x)\\
&\q=J_T^{0,\d}(s, x; \bar u^{\d}(\cd))-J_T^0 (s, x; \bar u(\cd))\nonumber\\
&\q\geq  J_T^{0,\d}(s, x; \bar u^\d(\cd))-J _T(s, x; \bar u^\delta(\cd))=\delta\dbE\int_s^T|\bar u^\delta(t)|^2dt.\nonumber
\end{align}
The above says that \begin{align}\label{31stronglimt1}\dbE\int_s^T|\bar u^\delta(t)|^2dt\leq \dbE \int_s^T|\bar u(t)|^2dt.\end{align}
Then there exists a $\d_k\downarrow 0$ such that $\{\bar u^{\delta_k}(\cd)\in \sU[s,T]\}$ admits a weak limit $\bar u^*(\cd)\in\sU[s,T]$ with 
\begin{align}\label{31stronglimt2}\dbE\int_s^T|\bar u^*(t)|^2dt\leq\liminf_{k\rightarrow\infty} \dbE\int_s^T|\bar u^{\delta_k}(t)|^2dt\leq \dbE \int_s^T|\bar u(t)|^2dt.\end{align}

By Proposition \ref{31themfiniteasy}, we know that $\{\bar u^{\delta_k}(\cd)\in \sU[s,T]\}$ is an asymptotic optimal sequence of controls for Problem (LQ)$_T$ at $(s, x)$. Since $J_T^0(s, x;\cd)$ is a convex and continuous functional in $\sU[s,T]$, the weak limit $\bar u^*(\cd)\in\sU[s,T]$ is also optimal. Combing \eqref{31stronglimt1} and \eqref{31stronglimt2}, it follows that 
$$\liminf_{k\rightarrow\infty} \dbE\int_s^T|\bar u^{\delta_k}(t)|^2dt=\dbE\int_s^T|\bar u^*(t)|^2dt.$$
Therefore, $\bar u^*(\cd)$ is also the strong limit of $\{\bar u^{\delta_k}(\cd)\}$.

Finally, by \eqref{31PdP1} and \eqref{31PdP2}, \eqref{31derivativePd} holds by taking $\bar u(\cd)=\bar u^*(\cd).$ The proof is complete.
\end{proof}

\begin{remark}\label{31remkustar}
It has been shown that $\bar u(\cd)\in\sU[s,T]$. In fact, we can further refine the measurability of the optimal control process. For $0\leq s<t\leq T$, define 
$$\cG^{s}_t:=\sigma\{W(r)-W(s):s\leq r\leq t\} \vee\cN.$$
where $\cN$ stands for the collection of all $\dbP$-null sets. 
Write the filtration 
$$\dbG^{s}:=\{\cG^{s}_{t}:t\geq s\}$$
Then for each $(s, x)\in\sD$, the optimal control  $\bar u^*(\cd)\in L^2_{\dbG^{s}}(s,T;\dbR^m)$.
\end{remark}

\section{The Bellman Principle and Range Condition}\label{sec:bp}

 In this section, we will prove the Bellman principle and a so-called range condition if Problem (LQ)$_T$  is finite.
 Such a result is essential to the proof  of generalized closed-loop solvability in the next section. 
 
The Bellman principle is not immediate in our setting, since our formulation allows only deterministic initial time–state pairs and defines the cost functional through an unconditional expectation rather than a conditional expectation. For the homogeneous problems studied in this paper, these two formulations are essentially equivalent by virtue of Remark \ref{31remkustar}. We adopt the former to maintain consistency with \cite{Sun-Li-Yong-2016}, but this choice requires an independent proof of the Bellman principle. The range condition plays an important role in passing to the limit from the Problem (LQ)$^{0,\d}_T$ to Problem (LQ)$^-_T$.

 We proceed with a sequence of lemmas. The first lemma concerns with the Bellman principle, which plays a fundamental role below.

\begin{lemma}[Bellman Principle]\label{33lemmabell}Suppose  Problem (LQ)$ ^0_T$ is finite with 
$V_T^0 (s, x)=\lan \bar P(s)x,x\ran.$ Then  for any $(s, x)\in \sD$ and $s_1\in[s,T]$, 
\begin{align*}&V_T^0 (s, x)=\inf_{u(\cd)\in\sU[s,s_1]}\dbE_{s, x}\Big(\int_s^{s_1}f(t,X(t),u(t))dt+V_T^0 (s_1,X(s_1))\Big).
\end{align*}
Consequently, $\bar P(\cd)$ is uniformly bounded. Here $$f(t,x,u):=\lan Q(t)x,x\ran +2\lan S(t)x,u\ran+\lan R(t)u,u\ran.$$
\end{lemma}
\begin{proof} Consider  Problem (LQ)$^{\d}_T$ whose cost functional is uniform convex.  Theorem \ref{thm1} implies that    
\begin{align*}&V_T^{0,\d}(s, x)=\inf_{u(\cd)\in\sU[s,s_1]}\dbE_{s, x}\Big(\int_s^{s_1}f(t,X,u)+\d\lan u,u\ran dt+V_T^{0,\d}(s_1,X(s_1))\Big)\\
&\q\geq \inf_{u(\cd)\in\sU[s,s_1]}\dbE_{s, x}\Big(\int_s^{s_1}f(t,X,u)dt+V_T^0(s_1,X(s_1))\Big).
\end{align*}
Take $\d\rightarrow 0^+,$ Proposition \ref{31themfiniteasy} yields that 
\begin{align*}&V_T^0 (s, x)=\lim_{\d\rightarrow 0^+}V_T^{0,\d}(s, x)\geq \inf_{u(\cd)\in\sU[s,s_1]}\dbE_{s, x}\Big(\int_s^{s_1}f(t,X,u)dt+V_T^0 (s_1,X(s_1))\Big).
\end{align*}

By Lemma \ref{31lemmmincontol} and Remark \ref{31remkustar}, for any $(s_1, x),$ there exists $\bar u^\e(\cd; s_1, x)\in\sU[s_1,T]\cap L^2_{\dbG^{s_1}}(s,T;\dbR^m)$ such that
\begin{align*}&J_T^0 (s_1, x;\bar u^\e(\cd; s_1, x))\leq V_T^0 (s_1, x)+\e|x|^2 \text{ and }\dbE\int_{s_1}^T|u^\e(t; s_1, x)|^2dt\leq C_\e |x|^2.
\end{align*}
Here $C_\e$ is a constant independent of $x.$
Then \begin{align*}\dbE\int_{s_1}^T|\bar u^\e(t; s_1,   X(s_1))|^2dt=C_\e
\dbE|X(s_1)|^2<\infty.\end{align*}
This says that $\bar u^\e(\cd; s_1,   X(s_1))\in\sU[s_1,T]$. By the definition of $V_T^0 (s_1, x)$ as the value function, it follows that 
\begin{align*}&V_T^0(s, x)\leq  \dbE_{s, x}\Big(\int_s^{s_1}f(t,X,u) dt+J_T^0 (s_1,    X(s_1);\bar u^\e(\cd; s_1,    X(s_1)))\Big)\\
&\q\leq  \dbE_{s, x}\Big(\int_s^{s_1}f(t,X,u) dt+V_T^0 (s_1,    X(s_1))+\e|X(s_1)|^2 \Big).
\end{align*}
for any $u(\cd)\in\sU[s,s_1].$  Let $\e\rightarrow 0^+$, the Bellman principle   holds.

It is evident that $P(\cd )$ is bounded above through the application of 0-control. Now let us show it is bounded below by contradiction. Suppose there exists a sequence of $\{s_k\}$ and $x$  such that $x^\top P(s_k )x\rightarrow\infty.$  

Let 
 $X(s_k):=X(s_k;0,x;0)=\Phi(s_k) x$,  where 
$$d\Phi=A  \Phi dt+C   \Phi dW\text{ with }\Phi(0)=I.$$
It follows that
\begin{align}\label{31bbounde below}\dbE[\Phi(t)^\top \Phi(t)]\geq \beta I,\text{ for some }\beta>0.\end{align}

Now applying  the Bellman principle,  it follows that 
\begin{align*}&V_T^0 (0, x)\leq \dbE_{0,x}\Big(\int_0^{s_k}f(t,X,0)dt+V_T^0 (s_k,   X(s_k))\Big)\\
&\leq \dbE_{0,x} \int_0^T|f(t,X,0)|dt+\dbE \lan \bar P(s_k) \Phi(s_k)x,\Phi(s_k)x\ran.
\end{align*}
By  \eqref{31bbounde below}, it is necessary that $V_T^0 (0, x)=-\infty$ for some $x\in\dbR^n$, which contradicts to the finiteness of Problem (LQ)$_T^0$. The proof is complete.
\end{proof}

Then let us derive some regularity results of the value function with respect to time.

\begin{lemma} Suppose that  Problem (LQ)$ ^0_T$ is finite with 
$V_T^0 (s, x)=\lan \bar P(s)x,x\ran.$ Then   $R+D^\top \bar P D\geq 0$ a.s.-$dt$. Moreover  $P(\cd )$ is right-continuous and is differentiable a.s.-$dt$.
\end{lemma}

\begin{proof}

Because $\bar P^\d\rightarrow P$ and $R+D^\top\bar  P^\d D\geq \d I$,  we have $R+D^\top \bar P D\geq 0$, a.s.-$dt$. Now let us show that $\bar P(s)$ is right-continuous.

Fix $s<T$  and $x\in\mathbb R^n$.
For any fixed $u(\cdot)\in\sU[s,T]$, let
$X^u(\cdot)=X(\cdot;s ,x;u(\cdot))$.
The standard estimates imply
$$
\mathbb E_{s,x }\sup_{s\leq t\leq s+\e}|X^u(t)-x|^2\longrightarrow0
\qquad\text{as }\e\rightarrow 0^+.
$$
Using the quadratic representation of the value function and
the boundedness of $P$, we obtain
\begin{align}
&|\dbE_{s,x }V_T^0 (s+\e,X^u(s+\e))-V_T^0 (s+\e ,x)| \nonumber\\
&\q\leq
K\dbE_{s,x }\Big(\big(|X^u(s+\e)|+|x|\big)\cdot|X^u(s+\e)-x|
\Big)\longrightarrow0.
\label{eq:RC-value-state}
\end{align}
Moreover, it is evident that 
\begin{equation}
\dbE_{s,x}\int_s^{s+\e}|f(t,X^u(t),u(t))|dt\longrightarrow0
\qq\text{as }\e\rightarrow 0^+.
\label{eq:RC-running-cost}
\end{equation}

Applying the Bellman  principle with zero control
on $[s,s+\e]$, we have
$$V_T^0 (s ,x)\leq
\dbE_{s,x }\Big[\int_s^{s+\e} f(t,X(t),0)dt+V_T^0(s+\e,X(s+\e))\Big].$$
By \eqref{eq:RC-value-state} and \eqref{eq:RC-running-cost},
\begin{equation}
V_T^0 (s ,x)
\le\liminf_{\e\rightarrow 0^+}V_T^0 (s+\e ,x).
\label{eq:RC-liminf}
\end{equation}

Conversely, fix $\varepsilon>0$, by finiteness of the problem,
there exists $u^\d(\cdot)\in\sU[s,T]$ such that
$$J_T^0 (s ,x;u^\d(\cdot))\leq V_T^0 (s ,x)+\d.$$
Let $X^\d(\cd)$ be the corresponding state.
Since the conditional continuation cost is bounded below by
the continuation value, we obtain
$$
\begin{aligned}
&V_T^0 (s ,x)+\d
\geq J_T^0 (s ,x;u^\d(\cdot))\geq
\dbE_{s,x }\Big[\int_s^{s+\e}f(t,X^\d(t),u^\d(t))dt+V_T^0 (s+\e,X^\d(s+\e))
\Big].
\end{aligned}
$$
Note that the selected control $u^\d(\cd)$ is fixed as $\e\rightarrow 0^+$.
Thus \eqref{eq:RC-value-state} and \eqref{eq:RC-running-cost} imply
$$
\limsup_{\e\rightarrow 0^+}V_T^0 (s+\e ,x)
\le V_T^0 (s ,x)+\d.
$$
Together with \eqref{eq:RC-liminf}, the arbitrariness of $\d>0$ implies that
$$
\lim_{\e\rightarrow 0^+}V_T^0 (s+\e ,x)=V_T^0 (s ,x).
$$
This is the right-continuity of $P(\cd )$.

Finally, let us show that $P(\cd )$ is differentiable a.s.-$dt.$ Taking the 0-control, applying the Bellman Principle, we have 
\begin{align}
x^\top P(s )x
\leq \dbE_{s,x }\Big[
\int_s^{s+h} \lan Q X,X\ran d r
+\lan \bar  P  (s+h)X(s+h),X(s+h)\ran
\Big].
\label{eq:BV-DPP}
\end{align}
Applying It\^o's formula to
$
r\longmapsto \lan \bar   P(s+h)X(r),X(r)\ran,
$
we have
\begin{align*}
&\dbE_{s,x }\lan \bar  P  (s+h)X(s+h),X(s+h)\ran
-x^\top\bar  P(s+h )x=
\dbE_{s,x }\int_s^{s+h}
X(r)^\top K(r,s+h)X(r)dr,
\end{align*}
where $
K(r ,s+h)
=\bar P(s+h )A(r)+A^\top(r) P(s+h )+C^\top(r) P(s+h )C(r).$

Combining this identity with \eqref{eq:BV-DPP}, we obtain
\begin{align}
&x^\top\big[\bar P(s )-\bar P(s+h )\big]x \leq
\dbE_{s,x }\int_s^{s+h}
X(r)^\top
\big[K(r,s+h)+Q  (r)\big]
X(r)dr.
\label{eq:BV-increment}
\end{align}
Set
$$
\ell(r)
:=
1+
|A(r)|+|C(r)|^2+|Q(r)|.
$$
By the coefficient assumptions, $\ell(\cd)\in L^1(0,T)$.
The boundedness of $P(\cd )$ and the finiteness of the regime space imply
$$
\bar P(s )-\bar P(s+h )
\leq CI\int_s^{s+h} \ell(r)dr.
$$
Therefore, the matrix-valued function
$$
F(t ):=\bar P(t )+CI\int_0^t \ell(r)dr
$$
is non-decreasing on $[0,T]$. For every fixed $z\in\mathbb R^n$, the function
$t\mapsto z^\top F(t )z$ is non-decreasing and bounded, and hence
has bounded variation. Consequently, it is differentiable a.s.-$dt$.
\end{proof}

\begin{corollary}\label{leftlimitP}
Suppose that  Problem (LQ)$ ^0_T$ is finite with 
$V_T^0 (s, x)=\lan \bar P(s)x,x\ran.$ Then for any $s\in (0,T],$ there exists a $$\bar P(s^-)=\lim_{t\uparrow s}\bar P(t)\leq \bar P(s).$$
\end{corollary}
\begin{proof} Note that $F(t)=\bar P(t )+CI\int_0^t \ell(r)dr$ is non-decreasing and $\ell\in L^1(0,T)$. The result holds directly.
\end{proof}

Now we are ready to show that the range condition holds if Problem (LQ)$ ^0_T$ is merely finite.

\begin{theorem}\label{31propImsubIm} Suppose that  Problem (LQ)$ ^0_T$ is finite with
$V_T^0 (s, x)=\lan\bar P(s)x,x\ran.$ 
 The range condition holds, i.e. $$\Im[B^\top {\bar P}+D^\top {{\bar P}}C+S]\subseteq \Im[R+D^\top {\bar P} D] \text{ a.s.-}dt.$$

\end{theorem}

\begin{proof} 

The 
Bellman principle  and It\^o's formula together implies that 
\begin{align}
&\lan \bar P(s) x,x\ran=\inf_{u(\cd)\in\sU[s,s+\e]}\dbE\Big(\int_s^{s+\e} f(t,X,u)dt+\lan \bar P(s+\e)X(s+\e),X(s+\e)\ran \Big)\nonumber\\
&\q=\lan \bar P(s+\e)x,x\ran+\inf_{u(\cd)\in\sU[s,s+\e]}\Big[\dbE\int_s^{s+\e}\Big\lan \Big(R  +D  ^\top \bar P(s+\e) D   \Big)u,u\Big\ran\nonumber\\
&\qq +2\dbE\int_s^{s+\e}\lan B  ^\top \bar P(s+\e) +C^\top   \bar P(s+\e) D  +S  )X,u\ran dt\Big] \nonumber\\
&\qq+\dbE\int_s^{s+\e} \Big\lan\Big(A  ^\top \bar P(s+\e)  +\bar P(s+\e) A  +C  ^\top \bar P(s+\e)C   +Q  \Big)X,X\Big\ran dt.\label{31dpp}
\end{align}
Note that $\bar P(\cd )$ is differentiable a.s.-$dt$.  For $s\in[0,T]$, where $\bar P(\cd )$ is differentiable,  taking $u(\cd)=v$ on $[s,s+\e]$ in \eqref{31dpp}, it follows that 
\begin{align} &0\leq \lan(\dot {\bar P}+A^\top \bar P+\bar P A+C^\top \bar P C+Q)x,x\ran\nonumber\\
&\q+2\lan (B^\top \bar P+D^\top \bar P C+S)x,v\ran +\lan (R+D^\top\bar  P D)v,v\ran,\label{Riccine}
\end{align}
for any $x\in\dbR^n$ and $v\in\dbR^m$. Then the range condition holds necessarily. The proof is complete.
\end{proof}


\section{Proof of Main Results}\label{sec:mr}

Now we are ready to prove  our main result.

\begin{proof}[Proof of  in  Theorem \ref{31thmclosedloopsolv}.]

(1). We only need to show the `if part'. Suppose  Problem (LQ)$_T^0$ is open-loop solvable. Consider  Problem (LQ)$_T^{0,\d}$ in Section \ref{sec:rp}. By Lemma \ref{31lemmmincontol}, 
let us define 
$$H(s):=\lim_{\d\rightarrow 0^+}\frac{\bar P^\d(s)-\bar P(s)}\d\text{ for each }s\in[0,T]. $$
We  notice that  $I+D^\top H D>0$ because $H\geq 0$. By \eqref{conRDPD}, the range condition  yields that \begin{align*}&\bar\Th_\d=-(R+\d I+ D^\top \bar P^\d D)^{-1}(B^\top \bar P^\d+D^\top \bar P^\d C+S)\rightarrow \bar \Th^\circ+\Phi=:\bar\Th^*, \text{ a.s.-}dt,\end{align*}
for some $\Phi(\cd)$ valued in $\dbR^{m\times n}$ with $[R+ D^\top \bar P D]\Phi=0$ a.s.-$dt$.

Now, 
we have obtained the  following \begin{align*}\begin{cases}&\bar\Th^\d(\cd)\rightarrow\bar\Th^*(\cd), \text{a.s.-$dt$},\\ 
&\ds\dbE\int_s^T|\bar u^\d(t)-\bar u^*(t)|^2dt\rightarrow 0,\\
&\ds\dbE\sup_{s\leq t\leq T}|X(t;s, x;\bar u^\d(\cd))-X(t;s, x;\bar u^*(\cd))|^2\rightarrow 0,\\
& \bar u^\d(\cd)=\bar \Th  ^\d(\cd) X(\cd;s, x;\bar u^\d(\cd))\text{ in }\sU[s,T].
\end{cases}
\end{align*}
Then there exists a subsequence $\{u^{\d_k}(\cd)\}$ admitting a pointwise limit $\bar u^*(\cd)$, a.s.-$dt\times d\dbP$. On the other hand,  we have the pointwise convergence $$\bar u^\d(\cd)=\bar \Th  ^\d(\cd) X(t;s, x;\bar u^\d(\cd))\rightarrow \bar \Th^*  (\cd) X(\cd;s, x;\bar u^*(\cd))\text{ a.s.-}dt\times d\dbP.$$
Therefore,
$\bar u^*(\cd)=\bar \Th^*  (\cd) X(\cd;s, x;\bar u^*(\cd))\text{ a.s.-}dt\times d\dbP.$
Provided  that $\bar u^*(\cd)\in\sU[s,T]$, $\bar \Th  (\cd) X(\cd;s, x;\bar u^*(\cd))\in\sU[s,T]$. \ms

(2). Provided Proposition \ref{proposition Thlocal}, for any $r\in[0,T)$, we know that $\bar\Th^*(\cd)\in L^2(r,r+\e_r;\dbR^{m\times n})$ for some $\e_r>0$. Then our results holds directly.

(3). If $P(s_k^-)=P(s_k)$, the result is trivial. Assume $P(s_k^-)\neq P(s_k)$.  Through the Bellman principle, for any $\e>0$, it follows that
\begin{align*}&V(s,x)=\dbE\Big(\int_s^{s_k-\e}f(t,\bar X,\bar u)dt+\lan P(s_k-\e) \bar X(s_k-\e), X(s_k-\e)\ran\Big)\\
&=\dbE\Big(\int_s^{s_k}f(t,\bar X,\bar u)dt+\lan P(s_k) \bar X(s_k), X(s_k)\ran\Big).\end{align*}
Take $\e\rightarrow 0^+$ and apply the continuity of $\bar X(\cd)$, we have 
$$\dbE\lan P(s_k^-) \bar X(s_k), X(s_k)\ran=\dbE\lan P(s_k-\e) \bar X(s_k), X(s_k-)\ran.$$
Since $P(s_k)\geq P(s_k^-)$ from Corollary \ref{leftlimitP}, our result holds.
\end{proof}

Until now, we have proven the equivalence of open-loop and generalized closed-loop solvability. We also derive a form of local wellposedness of the value function in time.  As we know that the value function may not be even continuous over time for merely open-loop solvable problems, the global well-posedness of DRE  on $[0,T]$ is not necessarily true.

In the rest of this section, we will derive the global well-posedness of  DRE \eqref{DRE}  from \eqref{31Riccatinot22} if some additional assumptions are provided.  Take $\d\rightarrow0^+$ in DRE \eqref{31Riccatinot22}. We notice that 
$$|A^\top {\bar P^\d}+{\bar P^\d}A+C^\top {\bar P^\d} C+Q|\leq K(1+|A|+|C|^2+|Q|).$$
Therefore, dominant convergence theorem yields that  $$A^\top {\bar P^\d}+{\bar P^\d}A+C^\top {\bar P^\d} C+Q\rightarrow A^\top {\bar P}+{\bar P}A+C^\top {\bar P} C+Q,~\text{ in }L^1(0,T;\dbS^n).$$ 
We also see that 
\begin{align*}&(\bar \Theta^\d)^\top(R+\d I+ D^\top \bar P^\d D)\bar \Theta^\d\longrightarrow\bar  \Th^*[R+ D^\top \bar PD]\bar \Th^*=\bar \Th^\circ[R+ D^\top \bar PD]\bar \Th^\circ,~\text{a.s.-}dt.\end{align*}

The gap for establishing the global DRE for $\bar P(\cd)$ lies in concluding the $L^1$-convergence  from the pointwise convergence, which requires the uniform  integrability of $$\Big\{(\bar \Theta^\d)^\top(R+\d I+ D^\top \bar P^\d D)\bar \Theta^\d:\d>0\Big\}.$$ While we notice that a version of local $L^2$-convergence always holds in the following proposition.
\begin{proposition} There exists a decomposition $\{[t_k,s_k):k=1,2\cdots\}$ of $[0,T)$.
On each $[t_k,s_k)$, for  any $\e\in(0, s_k-t_k)$, it follows that
$$\lim_{\d\rightarrow0^+}\int_{t_k}^{s_k-\e}|\bar\Th^\d(t)-\bar\Th^*(t)|^2dt=0.$$
\end{proposition}

\begin{proof}
Fix $s<T$. For each $j=1,\ldots,n$, let $u^{\delta,j}(\cd)$ be
the optimal control of the $\delta$-regularized homogeneous
problem with initial pair $(s,e_j)$, and let $X^{\delta,j}$
be its corresponding state.
$$
\bar u^{\delta,j}(\cd)=\bar \Theta^\delta(\cd)\bar X^{\delta,j}(\cd).
$$
By open-loop solvability and the convergence theorem for
regularized optimal controls, there exist optimal controls
$\bar u^j$ and corresponding states $\bar X^j$ such that
$$
\bar u^{\delta,j}(\cd)\longrightarrow\bar  u^j(\cd)
\q\text{in }\sU[s,T],
\q
\bar X^{\delta,j}(\cd)\longrightarrow \bar X^j(\cd)
\q\text{in }S^2_\dbF(s,T;\dbR^n).
$$

Define the matrix-valued processes
$$
\bar \cX^\delta
=(\bar X^{\delta,1},\cdots,\bar X^{\delta,n}),
\q
\cU^\delta
=(\bar u^{\delta,1},\cdots,\bar u^{\delta,n}),
$$
and define $\bar \cX^*,\bar \cU^*$ analogously using the
limiting states and controls. Then
$$
\bar \cU^\delta=\bar \Theta^\delta\bar \cX^\delta\text{ and } \bar \cU^*=\bar \Theta^*\bar \cX^*\text{ with }\bar \cX^\delta(s)=\bar \cX^*(s)=I_n.
$$
Moreover, it follows that 
$$\dbE\sup_{t\geq s}|\bar \cX^\d(t)-\bar \cX^*(t)|^2\rightarrow0\text{ and }\dbE\int_s^T|\bar \cU^\d(t)-\bar \cU^*(t)|^2dt\rightarrow0.$$

Set
$$
M^\delta(t)=\dbE\bar \cX^\delta(t),\q N^\delta(t)=\dbE\bar \cU^\delta(t),\q M(t)=\dbE\bar \cX^*(t),\q N(t)=\dbE\bar \cU^*(t).
$$
Since $\bar \Theta^\delta$ is deterministic,
$$
N^\delta(t)=\bar \Theta^\delta(t)M^\delta(t),\q N(t)=\bar \Theta^* M(t)
\quad\text{a.s.-}dt.
$$
Moreover, Jensen's inequality implies
$$
\sup_{t\in[s,T]}|M^\delta(t)-M(t)|^2
\leq
\dbE\sup_{t\in[s,T]}
|\mathcal X^\delta(t)-\mathcal X(t)|^2
\longrightarrow0
$$
and
$$
\int_s^T|N^\delta(t)-N(t)|^2dt
\leq
\dbE\int_s^T
|\mathcal U^\delta(t)-\mathcal U(t)|^2dt
\longrightarrow0.
$$
Fix any $\e\in(0,s_k-t_k)$. 
Given that $\bar\Th(\cd)\in L^2(t_k,s_k-\e;\dbR^{m\times n})$, we have $$\sup_{t\in[t_k,s_k-\e]}|M^{-1}(t)|\leq C_\e.$$
By the uniform convergence of $M^\delta(\cd)$ to $M(\cd)$, for all sufficiently
small $\delta$,
$$\sup_{t\in[t_k,s_k-\e]}| (M^\delta(t))^{-1}|\leq 2C_\e.$$
Consequently, as $\d\rightarrow 0^+$,
$$
\sup_{t\in[t_k,s_k-\e]}
|(M^\delta(t))^{-1}-M(t)^{-1}|=\sup_{t\in[t_k,s_k-\e]}|(M^\delta)^{-1}(M-M^\delta)M^{-1}|
\longrightarrow0.
$$

Since $N(\cd)\in L^2(t_k,s_k-\e;\dbR^{m\times n})$ and $M^{-1}(\cd)$  is
bounded on $[t_k,s_k-\e]$, $\bar \Theta^*\in L^2(t_k,s_k-\e;\dbR^{m\times n})$.
Hence,  as $\d\rightarrow 0^+$,
\begin{align*}
&\int_{t_k}^{s_k-\e}|\bar \Theta^\delta-\bar \Theta^*|^2dt
\leq
2\int_{t_k}^{s_k-\e}| N^\delta-N|^2dt\\
&\q+
\int_{t_k}^{s_k-\e}| N|^2dt
\sup_{t\in[t_k,s_k-\e]}
|(M^\delta(t))^{-1}-M(t)^{-1}|^2
\longrightarrow0.
\end{align*}
The proof is complete.
\end{proof}

To achieve the global convergence on $[0,T]$, the following proposition provides the exact path how generalized closed-loop solvability refines to closed-loop solvability.

\begin{proposition}\label{prop61} Suppose Problem (LQ)$_T^0$ is open-loop solvable with $V_T^0 (s, x)=\lan \bar P(s)x,x\ran.$

{\rm (1).} If $\{(\bar \Theta^\d)^\top(R+\d I+ D^\top \bar P^\d D)\bar \Theta^\d\}$ is uniform integrable on $[0,T]$,  then $\bar P(\cd)$ is the regular solution of  \eqref{31Riccatinot22}.

  {\rm (2).} If $\{|\bar \Theta^\d|^2\}$ is uniformly integrable on $[0,T]$, Problem (LQ)$_T^0$ is closed-loop solvable and  $\bar\Th^*(\cd)$ is   an $L^2$-optimal closed-loop strategy.

\end{proposition}

\begin{proof}

(1).
Provided the uniform integrability, we know that
\begin{align*}&(B^\top \bar P^\d+D^\top \bar P^\d C+S)^\top(R+\d I+ D^\top \bar P^\d D)^{-1}(B^\top \bar P^\d+D^\top \bar P^\d C+S)\\
&\rightarrow\bar \Th^\circ[R+ D^\top \bar PD]\bar \Th^\circ\text{ in $L^1(0,T;\dbS^{ n})$.}\end{align*}
Then $\bar P(\cd)$ is the regular  solution to  the DRE \eqref{DRE}. 

(2). Provided the uniform integrability, by Theorem \ref{thm1}, our result is true in a similar way. 
\end{proof}

The above proposition clearly states the gap from open-loop solvability to closed-loop solvability. In other words, for an  open-loop solvable problem, the full time $[0,T]$ can be decomposed into at most countable many regions on which the problem  is locally closed-loop solvable through the application of Bellman principle. Those uniform integrability assumptions secure the $L^1$ or $L^2$-convergence  across those singular times. This highlights the role of generalized closed-loop solvability  as a bridge to connect  open-loop and closed-loop solvability. 

\begin{remark}\rm
A stronger version of (2) in the above theorem is obtained in \cite{Sun-Li-Yong-2016}. The idea is that the uniform integrability together with the convexity of the functional (necessary for finiteness) will ensure the open-loop solvability, hence the assumption on open-loop solvability can be waived. Then the closed-loop solvability holds. We prefer to present the above easier version since it  provides the path how the open-loop solvability refines to the closed-loop solvability.\end{remark}

Finally, let us present an example for which DRE  equation admits a regular solution due to the uniform integrability as claimed in (1) of Proposition \ref{prop61}, while  an $L^2$-regular solution does not exist. This example shows that the absolute continuity of the value function is merely necessary for closed-loop solvability.

\begin{example}\label{31examplePdis4} \rm
Consider an ODE $$dX(t)=u(t)dt$$
with the cost functional ($T=1$)
$$J_T^0 (s,x;u(\cd))=\int_s^1\Big((1-t) u(t)+X(t)\Big)^2dt.$$
The value function $V _T(s,x)=0$ for $s\in [0,1]$ with $\bar P(\cd)=0.$  The above problem is open-loop solvable and thus generalized closed-loop solvable.

Consider the regularized problem where  the cost functional is uniform convex and the homogeneous $L^2$-feedback strategy exists uniquely.  the differential Riccati equation is stated as
\begin{align*}
\begin{cases}&\dot {\bar  P}^\d+1-( \bar P^\d+1-t)^2 / [\d+(1-t)^2] =0, ~t\in(0,1)\\
& \bar P^\d(1)=0.
\end{cases}
\end{align*} 

 Detailed calculation yields that
$$V_T^{0,\d}(s,x)\leq J_T^{0,\d}(s,x;0)=\int_s^1X(t)^2=(1-s)x^2.$$
Then  $$0\leq ( \bar P^\d+1-t)^2 / [\d+(1-t)^2]\leq 4( 1-t)^2 / [\d+(1-t)^2]\leq 4.$$
Therefore, $\{( \bar P^\d+1-t)^2 / [\d+(1-t)^2]:\d>0\}$ is uniformly integrable. This allows us to take $\d\rightarrow 0^+$ which leads to the differential Riccati equation 
\begin{align*}
\begin{cases}&\dot { \bar P}+1-( \bar P+1-t)^2  [(1-t)^2]^{\dagger} =0, ~t\in(0,1),\\
&  \bar P(1)=0.
\end{cases}
\end{align*} 
 One can  see that $\bar P(\cd)=0$ is a regular solution but not an $L^2$-regular solution  as $\bar\Th(t)=-(1-t)^{-1}.$
\end{example}

Before we finish this section, the following remark is worth here for the possibility of random initials.

\begin{remark}\rm In this paper, we consider initial pairs $(s,x)\in\sD$ with deterministic initial states $x\in\mathbb R^n$, following the formulation of \cite{Sun-Li-Yong-2016} to facilitate comparison with their results. This formulation can be extended to initial states $\xi\in L^2_{\cF_s}(\mathbb R^n)$ by replacing the expectation in the cost functional with conditional expectation given $\cF_s$ and defining the corresponding value as an essential infimum over admissible controls. For the homogeneous problem with deterministic coefficients, finiteness yields the quadratic representation
$$
V^0(s,\xi)=\langle P(s)\xi,\xi\rangle,\qquad\text{a.s.},
$$
where $P(s)$ is a deterministic symmetric matrix. Allowing random initial states provides a natural framework for the Bellman principle. In the homogeneous setting considered here, this extension is consistent with the deterministic-state value function, and the Bellman principle can therefore also be expressed using the formulation adopted in this paper.

For nonhomogeneous problems with  stochastic  drifts, the deterministic value function defined through unconditional expectation does not, in general, retain sufficient information for a Bellman principle expressed solely in terms of that function. A conditional formulation with random initial states is necessary for the dynamic programming in this setting. We will investigate this extension in a subsequent paper.
\end{remark}
\section{Infinite-Horizon Problems}\label{sec:IP}

It has been proven that open-loop and closed-loop solvability are equivalent in \cite{Sun-Yong-2018} for infinite-horizon problems. In this section, we revisit those problems by applying the idea adopted in the previous sections.
Throughout this section, we assume that $A,B,C,D,Q,S,R$ are deterministic constant matrices.

Consider the following $n$-dimensional homogeneous controlled  SDE: 
\begin{align}\label{31SDE4-122}\begin{cases}  &dX(t)=[AX(t) + Bu(t)  ]dt+[CX(t)  + Du(t) ]dW(t),~t\geq s, \\
 &X(s)=x.\end{cases}\end{align}
The  cost functional is defined by 
 \begin{align}\label{31cost4-12}  
  & J^0_\i(s,x;u(\cd)) =\dbE\Big[\int_s^\infty\lan QX(t),X(t)\ran +2\lan SX(t),u(t)\ran+\lan Ru(t),u(t)\ran dt\Big].
\end{align}

For any $u(\cd)\in\sU[0,\infty)$, the solution $X(\cd; s,x;u(\cd))$ to \eqref{31SDE4-122}  belongs to $S^2_\dbF(0,T;\dbR^n)$ for any $T\geq0.$ While this is not sufficient for the well-posedness of the cost functional  due to the lack of integrability over $[0,\infty)$. Therefore, we define the set of {\it admissible controls} by 
$$\sU^{ s,x}_{ad}[s,\infty):=\Big\{u(\cd)\in\sU[s,\infty)\Big|~\dbE\int_s^\infty|X(t; s,x;u(\cd))|^2dt<\infty\Big\}.$$
Here the superscript $(s,x)$ is adopted since this set is initial-dependent.

For any $u(\cd)\in\sU^{ s,x}_{ad}[s,\infty)$, the cost functional $J_\infty( s,x;u(\cd))$ is well-defined. This allows to consider the following optimization problem over $[s,\infty)$:\ms

\noindent {\bf Problem (LQ)$_{\i}$:} for any $( s,x)\in\sD $, find a $\bar u(\cd)\in \sU^{ s,x}_{ad}[s,\infty)$ such that 
\begin{equation*}
J^0_{\i}( s,x;\bar u(\cd))=\inf_{u(\cd)\in \sU^{ s,x}_{ad}[s,\infty)}J^0_{\i}( s,x; u(\cd))=:V^0_\i( s,x).
\end{equation*}
Here $V^0_\i(\cd)$ is called the {\it value function} for Problem (LQ)$_{\i}$. \ms

Problem (LQ)$_{\i}$ is usually referred to as {\it an infinite-horizon homogeneous stochastic LQ control problem.}  Similar as before,  our main effort lies in studying the open-loop solvability and closed-loop solvability, whose definitions will be presented below.

\begin{definition}[Open-Loop Solvability] Problem (LQ)$_{\i}$ is called {\it open-loop solvable   at $( s,x)\in\sD$},
if there exists a $\bar u(\cd)\in \sU^{ s,x}_{ad}[s,\infty)$ such that 
$$V^0_\i( s,x)=J^0_\i( s,x;\bar u(\cd)).$$
Problem (LQ)$_{\i}$ is called {\it open-loop solvable solvable}  if it is open-loop solvable at all $( s,x)\in\sD$.

\end{definition}

To introduce the closed-loop solvability, we need the following definition. We denote the system in \eqref{31SDE4-122} by $[A,C;B,D]$.
 \begin{definition}[Stationary Closed-Loop Strategy and Stabilizer] \rm
 (1). $\Th\in\dbR^{m\times n}$ is called a {\it stationary closed-loop strategy}. Correspondingly, write its realized control by 
 $u^\Th(\cd)$. Note that $X(\cd; s,x;u^{\Th}(\cd))\in S^2_\dbF(s,T;\dbR^n)$ for any $T>0.$
 
 (2). $\Th\in\dbR^{m\times n}$ is called a {\it stabilizer} for $[A,C;B,D]$ if 
$$X(\cd; s,x;u^\Th(\cd))\in L^2_{\dbF}(s,\infty;\dbR^n)\text{ for each }( s,x)\in\sD.$$
Write the set of all stabilizers by $\BS[A,C;B,D]$.
 \end{definition}
 With this, we can define the closed-loop solvability.

\begin{definition}[Closed-Loop Solvability]\rm 
Problem (LQ)$_{\i}$ is called {\it closed-loop solvable},
if there exists a $\Th\in\BS[A,C;B,D]$ such that for all  $( s,x)\in\sD$,
$$V^0_\i( s,x)=J^0_\i( s,x; u^\Th(\cd)).$$
\end{definition}

For the comparison with Problem (LQ)$_T$, we also introduce the  {\it generalized closed-loop solvability.}

\begin{definition}[Generalized Closed-Loop Solvability] \rm
 Problem (LQ)$^0_{\i}$ is called {\it generalized closed-loop solvable},
if there exists a $\Th\in\dbR^{m\times n}$ such that  at any $( s,x)\in\sD$,  there exists an optimal open-loop control $\bar u(\cd)\in \sU_{ad}^{ s,x}[s,\infty)$ such that 
$$ \bar u(\cd)=\Th X( s,x;\bar u(\cd))\text { in }\sU_{ad}^{ s,x}[s,\infty).$$

\end{definition}

The idea for the above definition is similar as before: to consider the $L^2$-integrability of control. While the specific structure of the optimal control problems, we can show that those two formulations of closed-loop solvability are equivalent. The following is the result.

\begin{proposition}Problem (LQ)$^0_{\i}$ is   generalized closed-loop solvable if and only if it is closed-loop solvable.
\end{proposition}
\begin{proof} `if part': this is trivial from the definitions.

`only if part': At any $(s,x)\in\sD$, suppose $\bar u(\cd)=\Th(\cd) X( s,x;\bar u(\cd))\in \sU_{ad}^{ s,x}[s,\infty)$. For any $T>s,$ we see that $\Th$ is also $L^2$-feedback strategy on $[s,T]$. This yields that 
$\bar u(\cd)=u^{\Th}(\cd)\text{ in }\sU[s,T].$ Since $\bar u(\cd)\in \sU_{ad}^{ s,x}[s,\infty)\subseteq \sU[s,\infty)$, so is $u^\Th(\cd)$. Therefore, $\Th$ is a stabilizer. The proof is complete.
\end{proof}

The above proposition is not surprising because it has been established that  the closed-loop  and open-loop solvability are equivalent for infinite-horizon problems. While the new framework gives a straightforward reason for this: $\Th$, admissible closed-loop strategy on $[0,T]$, is an $L^2$-closed-loop strategy for any $T>0$, since it is invariant in time.

\section{Concluding Remarks}\label{sec:cr}

This paper develops a new framework of generalized closed-loop solvability for stochastic LQ control problems on finite-horizon. It has been shown that the open-loop and generalized closed-loop solvability are equivalent. The new framework justifies  that the previous formulation on closed-loop solvability can be substantially refined to remove the constraint of $L^2$-integrability   of closed-loop strategies. For an  open-loop solvable problems, the full time $[0,T]$ can be decomposed into at most countable many regions on which the problem  is locally closed-loop solvable in the classical sense.  Our results provide a unified theory for  open-loop and closed-solvability for stochastic LQ control problems in finite and infinite-time horizons.

Even our study is focused on the standard homogeneous problems in this paper, the theory has a  great potential in other complex problems such as non-homogeneous problems, problems with switching,  differential games and so on. The authors hope to report more results on this direction in subsequent works.
%
%
%

\appendix

\section{Moore-Penrose pseudoinverse}\label{sec:app}

Let us recall
some results about the Moore-Penrose pseudoinverse.
For $M\in\dbS^n_{\geq 0}$, let its orthogonal decomposition be
 $$M=U^\top \text{diag}(\lambda_1,\cdots,\lambda_k,0,\cdots,0)U,$$ with $U^\top U=I$ and $\lambda_i>0$.
 Define  $$M^\dagger:=U^\top \text{diag}(\lambda_1^{-1},\cdots,\lambda_k^{-1},0,\cdots,0)U\in \dbS^n_{\geq 0},$$
 The $M^\dagger$ is said to be
 the {\it Moore-Penrose Pseudoinverse} of $M$. Because  we only consider an  inverse for non-negative definite symmetric matrices, such an easier version is adopted here only.
 Note
$$M^\dagger M M=MM^\dagger M=MMM^\dagger=M.$$
For $H\in\dbR^{n\times m}$, define
\begin{equation}\label{defIM}\Im[H]:=\{Hx:x\in\dbR^m\},\q \Ker[H]:=\{x\in\dbR^m: Hx=0\}.\end{equation}
The following lemma  is needed.
\begin{lemma}  Suppose  $R\in\dbS^n_{\geq 0}$, $H>0$ and $K,L\in\dbR^{n\times k}$ with $\Im[K]\subseteq \Im[R]$. Then
\begin{align}\label{conRDPD}&\lim_{\delta\rightarrow 0}(R+\d H+o(\d))^{-1}(K+\d L+o(\d))\\
&\q=R^\dagger K\!+\!(I\!-\!R^\dagger R)[(I\!-\!R^\dagger R)H(I\!-\!R^\dagger R)]^{\dagger}(I\!-\!R^\dagger R)(L\!-\!HR^\dagger K). \nonumber\end{align}

\end{lemma}

\begin{proof}
Write $\Lambda= \text{diag}(\lambda_1,\cdots,\lambda_k)$ with $\lambda_i>0$. Suppose that
$$R=U^\top\begin{pmatrix}\Lambda&0\\0&0\end{pmatrix}U,$$ with $U^\top U=I$.
Then \eqref{conRDPD} is equivalent to
\begin{align*}&\lim_{\delta\rightarrow 0}\left(\begin{pmatrix}\Lambda&0\\0&0\end{pmatrix}+\d U H U^\top +o(\d)\right)^{-1}\left(UKU^\top+\d U LU^\top+o(\d)\right)\nonumber\\
&\q=\begin{pmatrix}\Lambda^{-1}&0\\0&0\end{pmatrix}  UKU^\top+\begin{pmatrix}0&0\\0&I\end{pmatrix}\left[\begin{pmatrix}0&0\\0&I\end{pmatrix}UHU^\top\begin{pmatrix}0&0\\0&I\end{pmatrix}\right]^{\dagger}\\
&\qq\qq\qq\qq\times\begin{pmatrix}0&0\\0&I\end{pmatrix} \left(U LU^\top-UHU^\top\begin{pmatrix}\Lambda^{-1}&0\\0&0\end{pmatrix} UKU^\top \right). \end{align*}
For $M=H,K,L$, define $\widetilde M=U MU^\top $ with
$\widetilde M=\begin{pmatrix}\widetilde M_{11}&\widetilde M_{12}\\\widetilde M_{21}&\widetilde M_{22}\end{pmatrix}.$
Because $\Im[K]\subset \Im[R]$, we have $\widetilde K_{22}=0$ and $\widetilde K_{21}=0.$ Because $H>0$ (which implies  $\widetilde H>0$), $\widetilde H_{22}>0$.
Detailed calculation yields that both
the left-hand side
and the right-hand side
are equal to
$$\begin{pmatrix}
\Lambda^{-1}\widetilde K_{11}
&
\Lambda^{-1}\widetilde K_{12}
\\
\widetilde H_{22}^{-1}
\left(
\widetilde L_{21}
-\widetilde H_{21}\Lambda^{-1}\widetilde K_{11}
\right)
&
\widetilde H_{22}^{-1}(\widetilde L_{22}-H_{21}\Lambda^{-1}\widetilde K_{11})
\end{pmatrix}$$
The proof is complete.
\end{proof}
\end{document}